\documentclass{amsart}
\usepackage{graphics,epsfig}
\usepackage{xcolor}
\usepackage{amssymb}
\usepackage{amsthm}
\usepackage{marginnote}
\usepackage{color}
\usepackage{lineno}

\newtheorem{definition}{Definition}[section]

\newtheorem{remark}[definition]{Remark}
\newtheorem{proposition}[definition]{Proposition}
\newtheorem{lemma}[definition]{Lemma}

\newtheorem{theorem}[definition]{Theorem}

\begin{document}

\bigskip

\title{Envelope of intermediate lines of a plane curve}

\author[]{Ady Cambraia Jr.}

\address[Ady]{Departamento de Matem\'atica, Universidade Federal de Vi\c cosa, Brazil}
\email{ady.cambraia@ufv.br}

\author[]{Mostafa Salarinoghabi}
\address[Mostafa]{%
	Departamento de Matem\'atica, Universidade Federal de Vi\c cosa, Brazil.
}
\email{mostafa.salarinoghabi@ufv.br}

\author[]{Diego Trindade}
\address[Diego]{%
	Departamento de Matem\'atica, Universidade Federal de Vi\c cosa, Brazil.
}
\email{diego.trindade@ufv.br}

\thanks{The authors wish to express their gratitude to Richard Morris for his help in construction of figures. Also the authors thanks to Marcos Craizer for his valuable suggestions. The second author wants to thank PNPD-CAPES for financial support during the preparation of this paper.}

 
\begin{abstract}
	
	For a pair of  points in a smooth closed convex planar curve $\gamma$, its mid-line is the line containing its mid-point and the intersection point of the corresponding pair of tangent lines. It is well known that the envelope of the mid-lines ($EML$) is formed by the union of three affine invariants sets: Affine Envelope Symmetry Sets ($AESS$); Mid-Parallel Tangent Locus ($MPTL$) and Affine Evolute of $\gamma$. In this paper, we generalized these concepts by considering the envelope of the intermediate lines. For a pair of points of $\gamma$, its intermediate line is the line containing an intermediate point and the intersection point of the corresponding pair of tangent lines. Here, we present the envelope of intermediate lines ($EIL$) of the curve $\gamma$ and prove that this set is formed by three disconnected sets when the intermediate point is different from the mid-point: Affine Envelope of Intermediate Lines ($AEIL$); the curve $\gamma$ itself and the Intermediate-Parallel Tangent Locus ($IPTL$). When the intermediate point coincides with the mid-point, the $EIL$ coincides with the $EML$, and thus these sets are connected. Moreover, we introduce some standard techniques of singularity theory and use them to explain the local behavior of this set.
\end{abstract}
	


\maketitle

\section{Introduction}\label{intro}
Consider two different points $p_1$, $p_2$ of a smooth planar curve $\gamma$ which has non-parallel tangent lines at $p_1$ and $p_2$. The mid-line is the line that passes through the mid-point $M$ of $p_1$ and $p_2$ and the intersection of the tangent lines at $p_1$ and $p_2$. If these tangent lines are parallel, the mid-line is the line which passes through the point $M$ and is parallel to the both tangent lines. When $p_1=p_2$, the mid-line is just the affine normal at the point. The envelope of these mid-lines is an important affine invariant set associated with the curve, which has been studied by many authors  (see for instance \cite{giblin,Holtom,Warder,Bruce}). In \cite{ady1,ady2} the authors presented a generalization of these concepts for hypersurfaces in $\mathbb{R}^n$.

The {\it Envelope of Mid-Lines}, (denoted by $EML$), of the curve $\gamma$ is divided into three parts: The {\it Affine Envelope Symmetry Set} (denoted by $AESS$), the {\it Mid-Parallel Tangent Locus} (denoted by $MPTL$) and the {\it Affine Evolute}, which corresponded to the pairs of points $(p_1,p_2)$ with non-parallel, non-coincident parallel and coincident tangent lines, respectively. 
The geometry of these curves has also been extensively studied by many mathematicians (see for example \cite{giblin},\cite{Holtom}).

Unsurprisingly, the envelope of the tangent lines to $\gamma$ contains -at least- $\gamma$ itself. It is natural to ask what lies between the envelope of tangents and the envelope of mid-lines. Consider two disjoint points $p_1$, $p_2$ of the curve $\gamma$ with non-parallel tangent lines. An intermediate point of the points $p_1$ and $p_2$ is the point $M_\alpha=(1-\alpha)p_1+\alpha p_2$ with $\alpha \in [0,1]$. The intermediate line is the line which passes through the intersection point of the tangent lines  at $p_1$ and $p_2$ and the point $M_{\alpha}$ (Figure \ref{fig:Inter} illustrates an intermediate line for the points $p_1$ and $p_2$). 
\begin{figure}[tp]
	\includegraphics[scale=0.4]{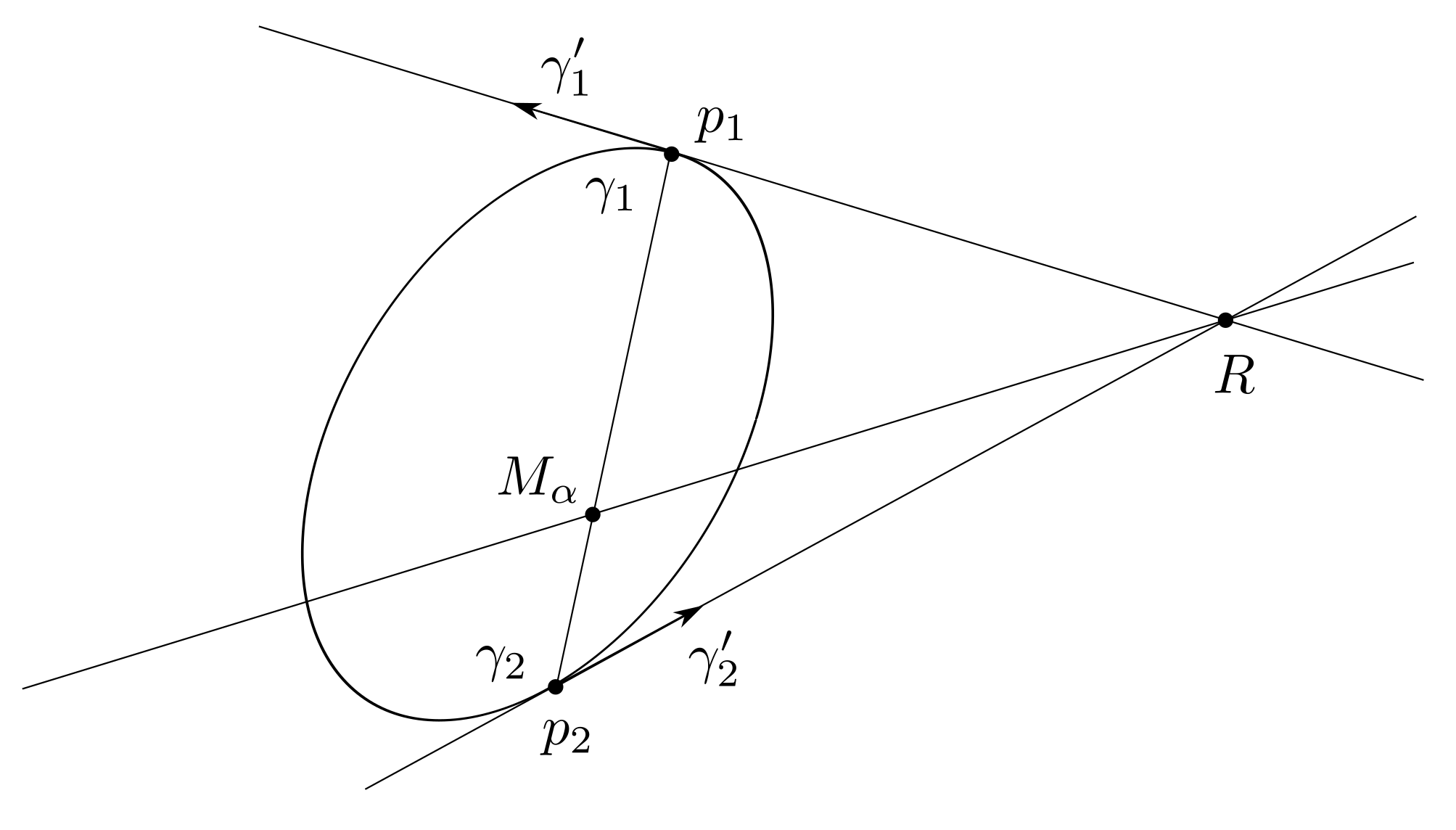}
	\caption{Intermediate line of the points $p_1$ and $p_2$.}\label{fig:Inter}
\end{figure}
When $\alpha=0$ or $\alpha=1$ the intermediate line corresponds to the tangent lines in points $p_1$ and $p_2$ respectively. When $\alpha=1/2,$ the intermediate line corresponds to the mid-line.
Naturally, the {\it Envelope of Intermediate Lines} (denoted by $EIL$) is also divided into three subsets:  The \textit{Affine Envelope of Intermediate Lines} (denoted by $AEIL$), which corresponds to the pairs of points $(p_1,p_2)$ with non-parallel tangent lines; the \textit{Intermediate-Parallel Tangent Locus} (denoted by $IPTL$) is consonant with the pairs of disjoint points $(p_1,p_2)$ with parallel tangent lines; and another curve called \textit{Coincident Tangent Lines}, (denoted by $CTL$), which corresponds to the limit case of coincident points.
This article aims to explain how some basic techniques of Affine Geometry and Singularity Theory enable us to say exactly what happens with the envelope of intermediate lines when $\alpha$ varies in $[0,1]\setminus \{1/2\}$. Precisely speaking, we will obtain the accurate conditions for which the cusps appear and disappear for a fixed $\alpha$ and, crucially, explaining the contribution of inflexions. 

We demonstrate that, unlike the case of $\alpha=1/2$ in which $AESS$ and $MPTL$ meet at their ordinary cusp singularities, when $\alpha\neq 1/2,$ $AEIL$ and $IPTL$ are disjointed.

The Affine Evolute of a planar curve is the limit of the $EML$, when $p_1=p_2$. In this paper we study the corresponding limit set $CTL$ and prove that this set is the curve $\gamma$ itself, when $\alpha\neq 1/2$.

The paper is organized as follows: In section \ref{sec:preliminaries}, we review some basic concepts of the Affine Differential Geometry of curves in $\mathbb{R}^2$. In section \ref{sec:AEIL}, we study the $AEIL$, giving necessary and sufficient conditions for a pair $(p_1,p_2)$ to contribute to this set. 
In section \ref{sec:local}, we investigate the local behavior of the curves $AEIL$ and $IPTL$ and prove that the set $CTL$ is the proper curve $\gamma$. In this section, we also study non-oval smooth curves. 


\section{Preliminaries}\label{sec:preliminaries}

In this section, we present the basic concepts of Affine Differential Geometry of planar smooth curves. We refer to \cite{Nomizu,Buchin} for more details on this subject.

Let $\gamma: [0,1] \longrightarrow \mathbb{R}^2$ be a $\mathcal{C}^{\infty}$ plane curve, without self-intersections or cusps, parametrized by $t$. The planar affine differential geometry is mainly focused on defining a new parametrization, $s$, which is an affine-invariant. The simplest affine-invariant parametrization $s$ for the curve $\gamma$ satisfies the following relation, 
\begin{equation}\label{pcaa}
[ \gamma_s, \gamma_{ss}]=1,
\end{equation}

\noindent where $[\,,\,]$ is the notation for determinants. 
When a curve satisfies equation \eqref{pcaa}, we say it this
parameterized by affine arc length parametrization $s$.

The vectors $\gamma_s$ and $\gamma_{ss}$ are the affine tangent and the affine normal of the curve $\gamma$, respectively.

The parameters $s$ and $t$ are related by 
$$ [\gamma_t,\gamma_{tt}]=\left[\gamma_s\,s_t\,,\,
\gamma_{ss}\,s_t^2\,+\,\gamma_s\,s_{tt}\right]=s_t^3[\gamma_s,\gamma_{ss}]=s_t^3,$$
therefore,
$$ \dfrac{ds}{dt}= [\gamma_t,\gamma_{tt}]^{\frac{1}{3}}.$$
By differentiating the equation \eqref{pcaa}, we obtain
$$ [\gamma_s, \gamma_{sss}]=0 \Rightarrow
\gamma_{sss}+ \mu(s)\gamma_s=0,$$ for some $\mu(s) \in \mathbb{R}.$
The function $\mu(s)$ is called the affine curvature and is
the simplest non-trivial affine differential invariant, defining 
up to an affine transformation. Notice that
$$ [\gamma_{s}, \gamma_{sss}]=0  \Rightarrow \gamma_{sss}=-\mu(s)\gamma_s,$$ therefore, we conclude that
$$\mu(s)=[\gamma_{ss}, \gamma_{sss}].$$

\begin{theorem}[\cite{Nomizu}] \label{thm1}
	Plane curves have constant affine curvature if and only if they are conic sections.
\end{theorem}


\begin{proposition}\label{Prop:Affine Normal}
	 Let $\gamma: \mathbb{R} \longrightarrow \mathbb{R}$ be a regular plane curve parametrized by an arbitrary parameter $t$. The affine normal $\xi(t)$ at $t$ is given by:
	\begin{eqnarray}\label{eq0}
	\xi(t)=\kappa^{-\dfrac{2}{3}}\gamma_{tt}-\dfrac{1}{3}\kappa_t\kappa^{-\dfrac{5}{3}}\gamma_t
	\end{eqnarray}
\end{proposition}
%
%
%
%
%
The affine curvature of a planar curve $\gamma$ parametrized by an arbitrary parameter is given in the next proposition.

\begin{proposition}\label{Prop:Affine Curvature}
	Let $\gamma$ be a smooth plane curve without inflection points pa\-ra\-me\-tri\-zed by an arbitrary parameter $t$. The affine curvature of $\gamma$ is given by
	\begin{eqnarray} \label{eq00}
	\mu = \dfrac{1}{9}\left(3\kappa\kappa_{tt}-5\kappa_t^2+9\kappa[\gamma_{tt},\gamma_{ttt}]\right)\kappa^{-\frac{8}{3}},
	\end{eqnarray}
	where $\kappa=[\gamma_t,\gamma_{tt}],$ denotes the Euclidean curvature of $\gamma$ and the indexes denote the derivative in respect to the given parameter. 
\end{proposition}
\begin{proof}
	Note that $s_t=\kappa^{\frac{1}{3}}$ and also
	$\gamma_s=\gamma_t\kappa^{-\frac{1}{3}}$. Now, by calculating
	$\gamma_{ss}$, $\gamma_{sss}$ and using the fact that
	$\kappa_t=[\gamma_t,\gamma_{ttt}]$, one can conclude that
	$\mu=[\gamma_{ss},\gamma_{sss}].$
\end{proof}
Consider the plane curve $\gamma$ in the Monge form parametrization without Euclidean inflection at the origin, that is,

$$\gamma(t)=\left(t, \displaystyle\frac{1}{2}a_2t^2+ \cdots +\dfrac{1}{k!}a_kt^k+g(t)t^{k+1}
\right),$$ with $a_i \in \mathbb{R}$, $a_2\neq0$ and $g$ a smooth function. Using the Proposition \ref{Prop:Affine Curvature}, the affine curvature of $\gamma$ at the origin is
$$\mu(0)=\dfrac{3a_2a_4-5a_3^2}{9a_2^{\frac{8}{3}}}.$$

This means that the affine curvature depends on $4$-jet of the curve $\gamma$ (see \cite{giblin4} for further information on jet space).
Throughout this paper we deal with the envelope of a family of lines defined on the curve $\gamma$. To be more precise, we have the following definition.

\begin{definition} \label{env}
	The envelope, or discriminant, of a family $$F: \mathbb{R}^n \times \mathbb{R}^r \rightarrow \mathbb{R} $$ of $n$ parameters, is the set: $$\mathcal{D}_F = \left\{x \in \mathbb{R}^r  \, | \, \exists \, u \in \mathbb{R}^n, F(u,x)=\frac{\partial F}{\partial u_i}(u, x)=0, i=1, \dots, n  \right\}.$$
\end{definition}
Consider a smooth plane curve $\gamma$. Take the points $p_1, p_2 \in \gamma$ and let $\gamma_1(t)$ and $\gamma_2(s)$ be the local parameterizations around $p_1$ and $p_2$ with $l_1$ and $l_2$ as their tangent lines of $\gamma$ at $p_1$ and $p_2$, respectively. If $l_1$ and $l_2$ are concurrent then consider $R$ as their intersection point. For $0<\alpha<1$, let $M_\alpha$ be the intermediate point of the segment joining $p_1$ and $p_2$, i.e., $M_\alpha=(1-\alpha)\gamma_1(t) +\alpha\gamma_2(s)$. 
\begin{definition}
	For a given oval smooth closed plane curve $\gamma$, the \textbf{intermediate line} of $\gamma$ is:
	\begin{enumerate}
		\item[$i.$] the line that passes through $R$ and $M_{\alpha}$ when $l_1$ and $l_2$ are concurrent with $t \neq s$;
		\item[$ii.$] the only line that passes through the point $M_{\alpha}$ and is parallel to the parallel tangent lines $l_1$ and $l_2$ with $t \neq s$;
		\item[$iii.$] the tangent line of $\gamma$ at $p_1$, when $t=s$. In the case of $\alpha=1/2$ (i.e. the mid-point), the intermediate line is the affine normal of $\gamma$ at $p_1$.
	\end{enumerate}
\end{definition}
We will use the co-normal maps to define the equation of the intermediate line . 
\begin{definition}
	Let $\gamma$ be a smooth curve. For $p\in\gamma$, let $\nu=\nu(p)$ the linear functional, such that $$\nu(p)(\xi)=1 \ \ \textrm{and} \ \ \nu(p)(X)=0,$$ where $\xi$ is the affine normal of $\gamma$ and $X$ is a tangent vector of $\gamma.$ The dif\-fe\-ren\-ti\-a\-ble map $\nu$ is called the {\it co-normal map.}
\end{definition}
For further details about the properties of co-normal map, see for instance \cite{Nomizu}.
%

Let $\nu_i,$ with  $i=1, 2,$ be the co-normal maps at $p_i,$ i.e., $\nu_i=\nu(p_i)$ with $$\nu_i(\xi_i)=1 \mbox{ \space  and \space  } \nu_i(\gamma_i')=0,$$
where $\xi_i$ is the affine normal at the point $p_i$.
\begin{lemma}\label{eri}
	The equation of the intermediate line is given by
	$$(1-\alpha)\nu_2(C)\nu_1(X-M_\alpha)+\alpha\nu_1(C)\nu_2(X-M_\alpha)=0,$$ where $\alpha \in (0,1)$ and $C=p_2-p_1$ is the chord joining the points $p_1$ and $p_2$.
\end{lemma}
\begin{proof}
	Let $\nu$ be the functional that annul the director vector of the intermediate line. Then:  $$\nu=a\,\nu_1 + b\,\nu_2.$$
	Notice that:
	\begin{eqnarray*}
	\nu(R-M_\alpha)=0 & \Longleftrightarrow & (a\,\nu_1 + b\,\nu_2)((1-\alpha)C + R - p_2)=0  \\
	&  & a\nu_1((1-\alpha)C)+a\nu_1(R-p_2)+b\nu_2((1-\alpha)C)=0.	
	\end{eqnarray*}
	
	In fact, $\nu_2(R-p_2)=0$, since $R-p_2$ is tangent to $\gamma_2$. As $R-p_2 = R-p_1 -C$, the last equation is equivalent to: 
	
	\begin{align*}
	 a\nu_1((1-\alpha)C) +a\nu_1(R-p_1) -a\nu_1(C)+b\nu_2((1-\alpha)C) &=0\\ 
	 -a\nu_1(\alpha C) + b\nu_2((1-\alpha)C)&=0,  
	\end{align*}
	because $R-p_1$ is tangent to $\gamma_1$. 
	
	Therefore, taking $a=(1-\alpha)\nu_2(C)$ and $b=\alpha\nu_1(C)$, we obtain the result.
	\end{proof}

Consider the family $F:\gamma_1 \times \gamma_2 \times \mathbb{R}^2  \longrightarrow \mathbb{R}$ given by: 
\begin{equation}\label{intermediate}
F(s,t,X)=(1-\alpha)\nu_2(C)\nu_1(X-M_\alpha)+\alpha\nu_1(C)\nu_2(X-M_\alpha),
\end{equation}
where $\alpha \in (0,1).$
For fixed points $p_1$ and $p_2$, and using the Lemma \ref{eri}, one can conclude that the equation of the intermediate line is $F(s,t,X)=0.$

According to the Definition \ref{env}, we intend to find the discriminant set $\mathcal{D}_F$ of the family $F$. More precisely, 
\begin{align}\label{discriminant}
\mathcal{D}_F=\left\{X\in\mathbb{R}^2 \, | \, \exists \, (s,t); F=F_s=F_t=0 \, \textrm{at} \, (X,s,t)\right\}.
\end{align} 
\section{Intermediate lines with non-parallel tangents}\label{sec:AEIL}

We begin with the following simple lemma:

\begin{lemma}\label{deco}
	The derivatives of co-normal maps $\nu_1$ and $\nu_2$ are given by:
	$$\nu_1^{\prime}=a\nu_1 + b\nu_2 \ \ \textrm{and} \ \ \nu_2^{\prime}=\bar{a}\nu_1 + \bar{b}\nu_2, $$
	where $a=\dfrac{\nu_1^{\prime}(\gamma_2^{\prime})}{\nu_1(\gamma_2^{\prime})}, b=\dfrac{\nu_1^{\prime}(\gamma_1^{\prime})}{\nu_2(\gamma_1^{\prime})}, \bar{a}=\dfrac{\nu_2^{\prime}(\gamma_2^{\prime})}{\nu_1(\gamma_2^{\prime})}$ and $\bar{b}=\dfrac{\nu_2^{\prime}(\gamma_1^{\prime})}{\nu_2(\gamma_1^{\prime})}.$
\end{lemma}

\begin{proof} Take the basis $\left\{\nu_1,\nu_2\right\}$ of the dual plane $\mathbb{R}_2.$ Thus, we can write the linear functional $\nu_1^{\prime}$ as a linear combination of the basis vectors, i.e., $\nu'_1=a\nu_1+b\nu_2.$ By applying $\nu'_1$ to $\gamma'_1$ we obtain $\nu'_1(\gamma'_1)= b\nu_2(\gamma'_1),$ which leads us to achieve 
\[b=\dfrac{\nu_1^{\prime}(\gamma_1^{\prime})}{\nu_2(\gamma_1^{\prime})}.\]
The other coefficients are obtained analogously.
\end{proof}

The next result allows us to locate pairs of points of $\gamma$ that contribute to the envelope of intermediate lines, through which we can relate the parameters $s$ and $t$.

\begin{theorem}\label{pre}
	The necessary and sufficient condition for which the pair $(s, t)$ determines one point of the envelope of intermediate lines is
	
	\begin{eqnarray}\label{eq1}
	\nu_1(C)=-\lambda\nu_2(C),
	\end{eqnarray}
	 where
	$$\lambda=\left(\dfrac{1-\alpha}{\alpha}\right)^{\frac{1}{3}}.$$
\end{theorem}

\begin{proof} Since $\nu_1(\gamma'_1(t))=\nu_2(\gamma'_2(s))=0$, so the derivative of the Equation \eqref{intermediate} with respect to $t$ is 	
\begin{align*}
	F_t= & (1-\alpha)\nu_2(-\gamma'_1)\nu_1(X-M_\alpha)+(1-\alpha)\nu_2(C)(\nu'_1(X-M_\alpha))+ \\  & +\alpha\nu'_1(C)\nu_2(X-M_\alpha)-\alpha \nu_1(C)\nu_2((1-\alpha)\gamma'_1).
\end{align*} 
Using Lemma \ref{deco}, we obtain:
	\begin{align}\label{eq2}
F_t = & -(1-\alpha)\nu_2(\gamma'_1)\nu_1(X-M_\alpha)+b\nu_2(C)\nu_2(X-M_\alpha) \nonumber \\  &-\alpha(1-\alpha)\nu_1(C)\nu_2(\gamma'_1)+aF.
	\end{align}
Similarly,
	\begin{align}\label{eq3}
	F_s = & \ \overline{a}\nu_1(C)\nu_1(X-M_\alpha)+\alpha\nu_1(\gamma'_2)\nu_2(X-M_\alpha) \nonumber \\ &- \alpha(1-\alpha)\nu_2(C)\nu_1(\gamma'_2)+\overline{b}F,
	\end{align}
By isolating $\nu_2(X-M_\alpha)$ from $F=0$ and replacing it in the equations $F_t=0$ and $F_s=0$, we obtain
\begin{eqnarray}\label{eq333}
	\nu_1(X-M_\alpha)=\dfrac{-\alpha^2\nu^2_1(C)\nu_2(\gamma'_1)     }{\alpha\nu_1(C)\nu_2(\gamma'_1) +b\nu^2_2(C)},
\end{eqnarray}
\begin{eqnarray}\label{eq444}
	\nu_1(X-M_\alpha)=\dfrac{\alpha(1-\alpha)\nu_1(C)\nu_2(C)\nu_1(\gamma'_2) }{\overline{a}\nu^2_1(C)-(1-\alpha)\nu_2(C)\nu_1(\gamma'_2)}.
\end{eqnarray}
For the rest of the proof, we	
 equate the equations \eqref{eq333} and \eqref{eq444}, and use the formulae of $b$ and $\bar{a}$ given in the Lemma \eqref{deco}.
	%
	%
	%
	%
	%
	%
	%
\end{proof}
\begin{remark}
	If $\gamma_1$ and $\gamma_2$  are parametrized by the affine arc length parametrization, then $\lambda=\beta.$ Thus, the condition given in the Theorem \ref{pre} is equivalent to 
$$[\gamma_1^{\prime}+\lambda\gamma_2^{\prime}, \gamma_2-\gamma_1]=0,$$ 
where $\lambda=\left(\dfrac{1-\alpha}{\alpha}\right)^\frac{1}{3}.$ 
\end{remark}
\begin{theorem}\label{Theo:Solution of Discr} 
Let $F$ be the family of the intermediate lines given in \eqref{intermediate}. The discriminant of the family $F$ admits a solution if and only the equation \eqref{pre} holds. Moreover, the solution of the system is given by
	\begin{equation}
	\displaystyle X-M_\alpha= \frac{\alpha\nu_1(C)}{\alpha\nu_2(\gamma_1')\nu_1(C)+ b\nu_2^2(C)  }\left(  (1-\alpha)\nu_2(C)\gamma_1'- \frac{\alpha\nu_1(C)\nu_2(\gamma_1')}{\nu_1(\gamma_2')}\gamma_2'       \right),
	\end{equation}
	where $\alpha \in (0,1).$
\end{theorem}
\begin{proof}
	Since  $\gamma'_1$ and $\gamma'_2$ are not parallels, we can write $X-M_\alpha= \overline{A}\gamma'_1 + \overline{B}\gamma'_2.$ Thus, by applying  $\nu_2$ the latter equation and isolating $\nu_2(X-M_\alpha)$ in the equation of intermediate line, we have:
	
	$$\overline{A}=\dfrac{\nu_2(X-M_\alpha)}{\nu_2(\gamma'_1)}=\dfrac{-(1-\alpha)\nu_2(C)\nu_1(X-M_\alpha)}{\alpha\nu_2(\gamma'_1)\nu_1(C)} $$ 
	
	Using the expression  \eqref{eq333} :
	\begin{align*}
	\overline{A} = & \dfrac{(1-\alpha)\nu_2(C)}{\alpha\nu_2(\gamma'_1)\nu_1(C)}\left( \dfrac{\alpha^2\nu^2_1(C)\nu_2(\gamma'_1)}{\alpha\nu_1(C)\nu_2(\gamma_1')+b\nu_2^2(C)}\right)\\
	= & \dfrac{\alpha(1-\alpha)\nu_1(C)\nu_2(C)}{\alpha\nu_1(C)\nu_2(\gamma'_1) + b\nu_2^2(C)}.
	\end{align*}
	Similarly,
	$$\overline{B}=\dfrac{-\alpha^2\nu_1^2(C)\nu_2(\gamma'_1)}{\nu_1(\gamma'_2)(\alpha\nu_2(\gamma'_1)\nu_1(C)+b\nu^2_2(C))}.$$
\end{proof}
\begin{remark}
	\
	\begin{enumerate}
		\item  If the curve $\gamma$ is parameterized by the affine arc length, the envelope of the intermediate lines is given by
	    \begin{align*}
		 X - M_\alpha & = \frac{ \alpha\lambda[\gamma_2^{\prime}, C ] [\gamma_2^{\prime}, \gamma_1^{\prime}]      }{     \alpha\lambda[\gamma_2^{\prime}, \gamma_1^{\prime}]^2 + [\gamma_2^{\prime}, C]}\left( (1-\alpha)\gamma_1^{\prime} -\alpha\lambda\gamma_2^{\prime}\right).
		\end{align*} 
		\item Notice that when $\alpha=1/2$, we have the familiar set Affine Envelope Symmetry Set. More precisely,
		
		$$ X - M = \frac{1}{2}\left( \frac{ [\gamma_2^\prime,\gamma_1-\gamma_2] [\gamma_1^\prime,\gamma_2^{\prime}] }{     -2[\gamma_2^\prime,\gamma_1-\gamma_2] + [\gamma_1^\prime,\gamma_2^{\prime}]^2}\right) (\gamma_1^{\prime} -\gamma_2^{\prime}),$$  
		for further details about this set, see \cite{giblin}.
	\end{enumerate}
\end{remark}
\section{Local structure of the envelope of Intermediate lines}\label{sec:local} 
Let $\gamma$ be a $C^{\infty}$ closed plane curve. Also let $p_1=\gamma(t)$ and $p_2=\gamma(s)$ be two different points of $\gamma$ with $\kappa(p_1)\neq0$ and $\kappa^\prime(p_1)\neq0$, where $\kappa$ is the curvature at the point $p_1$. This means that $p_1$ is neither an inflection point nor a vertex point. Without loss of generality, we can assume that 
\begin{equation}\label{P1andP2}
\begin{array}{ccl}
p_1 & = &(t,f(t))=\left(t,\frac{1}{2}t^2+a_3t^3+a_4 t^4+a_5 t^5 + O_6(t)\right),\\
p_2 & =& (s,g(s))=\left(s, b_0+b_1 s+b_2 s^2 +b_3s^3+b_4s^4+b_5 s^5 +O_6(s)\right),
\end{array}
\end{equation}
where $b_0>0$, $a_3>0$, $b_1^2-b_0>0$ and $O_6(t)$ (resp. $O_6(s)$) represents the terms of order grater than or equal to 6 with respect to the parameter $t$ (resp. $s$).
The tangent lines at the points $p_1$ and $p_2$ can be parallel or concurrent.

Let $F(x,y,t,s)=l_1(t,s) x + l_2(t,s) y+l_3(t,s)$ be the equation of the intermediate line passing through the point $M_{\alpha}$ given in (\ref{intermediate}). The condition of existence of solution for the envelope of the family of such lines is that the determinant of the following matrix be zero,
\begin{equation}\label{Matrix}
M=\left[
\begin{array}{ccc}
F_x(0,0,t,s) & F_y(0,0,t,s) & F(0,0,t,s) \\
F_{sx}(0,0,t,s) & F_{sy}(0,0,t,s) & F_s(0,0,t,s) \\
F_{tx}(0,0,t,s) & F_{ty}(0,0,t,s) & F_t(0,0,t,s)
\end{array}\right].
\end{equation}
The determinant of the matrix (\ref{Matrix}) is  
\begin{align}\label{detM} \det(M)= \frac{1}{8} \alpha (\alpha -1) A(t,s) B(t,s),\end{align}
where $A(0,0)=0$ and $B(0,0)=8(2\alpha b_2-\alpha+1)b_0^3$ and we have 
\begin{align}\label{A}
	A(t,s)= t-2b_2 s + 3 a_3 t^2-3b_3 s^2 + O_3(t,s),
\end{align}
where $O_3(t,s)$ denotes terms with respect to the parameters $t$ and $s$ with the order greater than or equal to three.
\subsection{Concurrent tangent lines}
If the tangent lines at the points $p_1$ and $p_2$ are not parallel, then the determinant of the matrix (\ref{Matrix}) given in (\ref{detM}) is zero, if  $B(t,s)=0$ (note that in this case $A(t,s)\neq 0$). Since $B(0,0)=8(2\alpha b_2-\alpha+1)b_0^3$ and $b_0\neq 0$. the expression $B(t,s)$, locally, has zero around the origin, if $b_2=\frac{\alpha-1}{2\alpha}$. Using this assumption, one can write the parameter $t$ as a function of $s$, due to the implicit function theorem. To be more precise, we have
\begin{align}\label{t}
	t = \frac{2\alpha b_0 b_3+b_1(\alpha-1)}{(\alpha-1)(2a_3b_0+b_1)} s +O_2(s).
\end{align}
\begin{remark}
	Note that as the point $p_1$ is not a vertex point, then, without loss of generality, we can assume that $b_0, b_1\,\,\text{and}\,\,a_3>0$. Therefore, $2a_3b_0+b_1 \neq 0$.	
\end{remark}   
\begin{theorem}\label{theo:non_Paralel}
	For a plane curve $\gamma$, suppose that the points $p_1$ and $p_2$ given in (\ref{P1andP2}) have non-parallel tangents. We have: 
	\begin{itemize}
		\item[i)] If neither the point $p_1$ nor $p_2$ are inflection points, the condition of regularity of the envelope of the family of intermediate lines (i.e. the curve $AEIL$) is
		\begin{align*}
			\alpha & \neq \frac{b_0}{b_1^2},1,\\ 
			b_3 & \neq \frac{(\alpha-1)(-6\alpha a_3b_0b_1^2+4\alpha a_3b_0^2-3\alpha b_1^3-2a_3b_0^2)}{2\alpha b_0(6\alpha a_3b_0b_1+3\alpha b_1^2+2\alpha b_0-b_0)}, 
		\end{align*} 
		\item[ii)] If the point $p_1$ is an inflection point, then the curve $AEIL$ is regular at $s=0$ if $b_3\neq 0$, equivalently, if $p_2$ is an inflection point at the origin.
 	\end{itemize} 
\end{theorem} 
\begin{proof}
	i) Let the points $p_1$ and $p_2$ given in \eqref{P1andP2} not be inflection points. After substituting the parameter $t$, given in \eqref{t}, in the envelope of intermediate lines \eqref{intermediate}, we reach the following parametrization
		\begin{align}\label{F non-parallel}
			F(t,\alpha) = \left(\tilde{a}_0+\tilde{a}_1 s+O(s^2), \tilde{b}_0+\tilde{b}_1 s + O(s^2)\right),
		\end{align}
		where 
		\begin{align*}
			\tilde{a}_0 & = -\frac{\alpha b_1 b_0}{\alpha b_1^2-b_0},\\  
			\tilde{a}_1 & = -\dfrac{b_0 \delta}{(b_1^2\alpha^2-(b_1^2+b_0)\alpha+b_0)(2a_3b_0+b_1)(b_1^2\alpha-b_0)},\\
			\tilde{b}_0 &= -\frac{\alpha  b_0^2}{\alpha b_1^2-b_0},\\
			\tilde{b}_1 & = -\dfrac{b_1 \delta}{(b_1^2\alpha^2-(b_1^2+b_0)\alpha+b_0)(2a_3b_0+b_1)(b_1^2\alpha-b_0)},
		\end{align*}
		and $\delta$ is a constant term. Note that, as we assume that $b_1^2-b_0>0$, the denominator of $\tilde{a}_1$ and $\tilde{b}_1$ are not nulls for any real value of $\alpha\neq b_0/b_1^2$ and $\alpha\neq1$. The envelope \eqref{F non-parallel} is singular at $s=0$ if $\delta=0$. This is observed if:
		\[
		b_3 = \frac{\alpha-1}{2\alpha b_0}\frac{(-6\alpha a_3b_0b_1^2+4\alpha a_3b_0^2-3\alpha b_1^3-2a_3b_0^2)}{(6\alpha a_3b_0b_1+3\alpha b_1^2+2\alpha b_0-b_0)}.
		\]
		This completes the proof of item (i).
		
		ii) Now suppose that the curve $\gamma$ is non-oval and has inflection point at the point $p_1$. Therefore, we can assume that the points $p_1$ and $p_2$ have the following local parametrizations: $p_1=(t,a_3 t^3+O(t^4))$ and $p_2=(s,b_0+b_1 s+b_2 s^2+O(s^3))$, with $b_0b_1\neq 0$. In this case, the component that annuls in determinant of the matrix $M$ given in \eqref{Matrix} becomes	
		$$B(t,s)=16 \alpha b_0^3 b_2+O(t,s).$$
		Hence, $B(0,0)=0$ if and only if $b_2=0$. So, by using the Implicit Function Theorem, we obtain $t$ as function of $s$. To be more precise, we have 
		\[
		t = \dfrac{\alpha b_3}{(\alpha-1)a_3}s+O(s^2).
		\]
		Therefore the parametrization of the envelope \eqref{intermediate} is as follow
		\[
		F(s,\alpha) = \left(-\frac{b_0}{b_1}+\hat{a} s+O(s^2) , \hat{b} s + O(s^2) \right),
		\]
		where 
		\[
		\hat{a} \hat{b} = \dfrac{36\alpha b_3^2b_0^4}{b_1^5(\alpha-1)^2}.
		\]
		This means that $F(s,\alpha)$ is singular at the origin, if $b_3=0$. Equivalently, $F(s,\alpha)$ is regular at the origin, if $p_2$ is an inflection point (a schematic local behavior of $AEIL$ is given in Figure \ref{fig:InfNP}).   
\end{proof}
\begin{figure}[h]
	\includegraphics[scale=0.4]{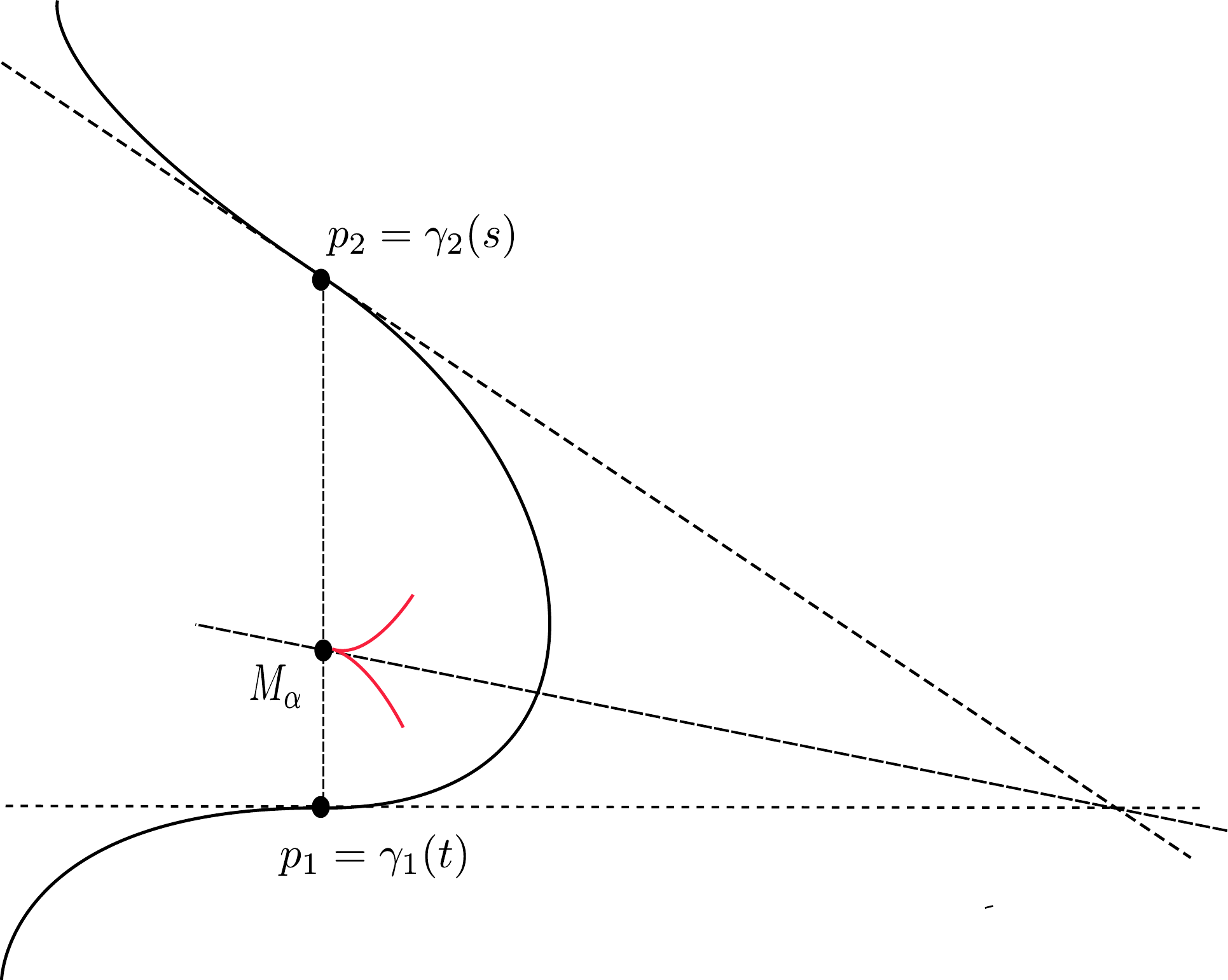}
	\caption{The $AEIL$ (red curve) of a non-oval curve $\gamma$ (black curve), when $p_1$ and $p_2$ are inflection points.}\label{fig:InfNP}
\end{figure}
\subsection{Non-coincident parallels}
Suppose that the points $p_1$ and $p_2$, given in \eqref{P1andP2}, have paralell tangent lines. Hence, we have $b_1=0$. Also, the equation $A(t,s)=0$ given in \eqref{A}, gives the relation between the parameters $t$ and $s$ in this case.

Using the Implicit Function Theorem for the expression (\ref{A}), one can obtain $t$ as a function of $s$, as follows 
\[ t = 2b_2 s + (-12a_3b_2^2+3b_3) s^2+ O_3(s). 
\] 

\begin{proposition}\label{prop:parallel}
	Let $\gamma$ be a smooth plane curve and $p_1$ and $p_2$, two points of $\gamma$, as given in (\ref{P1andP2}). If the tangent lines of $\gamma$ at the points $p_1$ and $p_2$ are parallel, then the $IPTL$ of the curve $\gamma$ 
	
	\begin{enumerate}
		\item passes through the point $M_\alpha(0,0)=(0,b_0\alpha)$.
		\item  is regular at $s=0$, if and only if $b_2\neq \dfrac{\alpha}{2(\alpha-1)}$.
		\item has an ordinary cusp singularity at $s=0$, if and only if  
		$$	b_2 = \frac{\alpha}{2(\alpha-1)} \,\,\text{and}\,\, b_3\neq \left(\dfrac{\alpha}{\alpha-1}\right)^2 a_3.$$
		\item has a $(3,4)$-cusp singularity at $s=0$ if and only if 
		$$	b_2 = \frac{\alpha}{2(\alpha-1)} \,\,,\,\, b_3 = \left(\dfrac{\alpha}{\alpha-1}\right)^2 a_3\,\,\,\text{and}\,\,\,
		b_4 \neq \left(\dfrac{\alpha}{\alpha-1}\right)^3 a_4.
		$$
	\end{enumerate}
\end{proposition} 
\begin{proof}
	The proof comes from a straightforward calculation since a plane curve $\theta$ is regular at $t=0$ if $\theta'(0)\neq 0$ and has an ordinary cusp (resp. a $(3,4)$-cusp), singularity at the origin if $\theta'(0)=0$ and $[\theta^{(2)}(0),\theta^{(3)}(0)]\neq 0$ (resp. if $\theta'(0)=[\theta^{(2)}(0),\theta^{(3)}(0)]=0$ and $[\theta^{(3)}(0),\theta^{(4)}(0)]\neq 0$), where $[\cdot,\cdot]$ stands for the determinant of the $2\times 2$ matrix defined by $\mathbb{R}^2$ vectors.
\end{proof}

In the next proposition, we investigate the case when there is an inflection point. Let $p_1=(t,a_3 t^3+O(t^4))$ and $p_2$ is the same as in \eqref{P1andP2}.

\begin{proposition}\label{prop:parallel_p1_Inflection}
	By the above assumptions,
	if one of the points $p_1$ and $p_2$ or both of them have inflection at the origin, then the curve $IPTL$ passes through the point $M_{\alpha}$, and is regular at the origin, when $\alpha\in (0,1)$.

 Moreover, for $b_2\neq 0$, the curve $\gamma$ has an inflection of order $k$ at $p_1$, if and only if $IPTL$ has an inflection of order $k$ at the origin, for $k=1,2$.       
\end{proposition}
\begin{proof}
	 Firstly, suppose that $p_2$ is not an inflection at $s=0$ i.e. $b_2\neq 0$.  Using the expression \eqref{A}, and by a straightforward calculation, one can write $s$ as a function of $t$. Therefore, we have the parametrization of the curve $IPTL$, as follows
	\[
	\left((1-\alpha) t + O(t^2), \alpha b_0 +a_3(1-\alpha) t^3+\frac{9\alpha a_3^2+4a_4 b_2(1-\alpha)}{b_2} t^4+O(t^5)\right). 
	\]  
	Now, suppose that the curve $\gamma$ has an inflection at the point $p_2$ as well. Therefore, using the relation $\nu_1(C)=-\beta\nu_2(C)$ where $\beta=\left((1-\alpha)/\alpha\right)^\frac{1}{3}$ (see Theorem \ref{pre}), we obtain:
	\[
	s= - \frac{a_3}{b_3}\left(\frac{\alpha-1}{\alpha}\right) t+O(t^2).
	\]  
By replacing this in \eqref{intermediate}, we obtain the parametrization of the envelope of the intermediate lines
\[
F(t,\alpha)= \left(\frac{h_1(\alpha)}{2\alpha b_3} t+O(t^2) , \alpha b_0 + \frac{h_2(\alpha)}{2\alpha^2 b_3^2} t^3 + O(t^4) \right),
\]
where $h_1$ and $h_2$ are smooth functions with respect to $\alpha$ and their $1$-jets are:
\begin{align*}
	j^1 h_1(\alpha) = a_3-2(2a_3+b_3)\alpha\quad,\quad	j^1 h_2(\alpha) = a_3^3-3a_3^2(2a_3-b_3)\alpha.
\end{align*}
Figure \ref{fig:InfP} illustrates the $IPTL$ of the curve $\gamma$, when $p_1$ and $p_2$ are inflections with parallel tangents.
\end{proof}
\begin{figure}[tp]
	\includegraphics[scale=0.4]{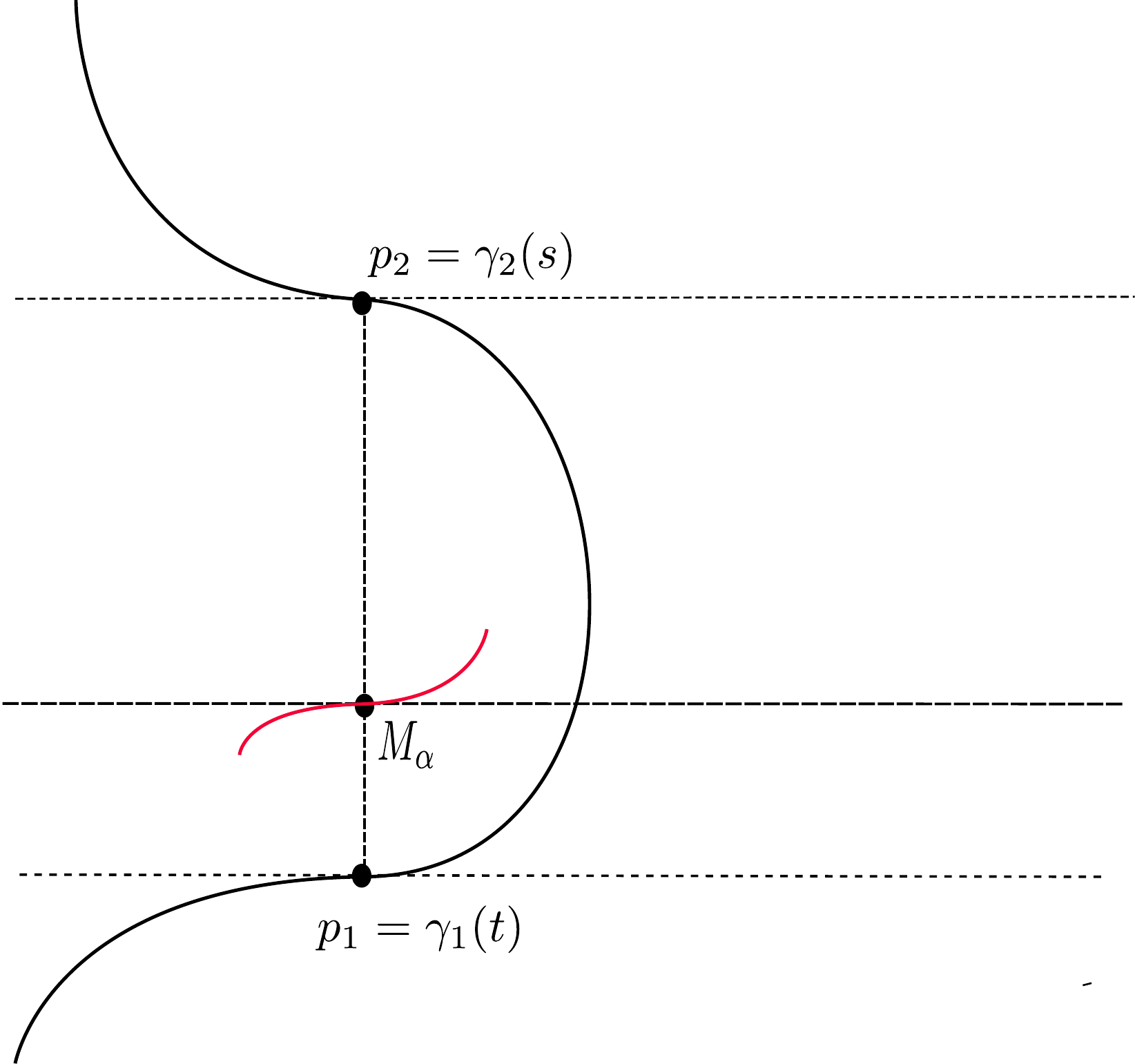}
	\caption{The $IPTL$ (red curve) of the curve $\gamma$ (black curve) when $p_1$ and $p_2$ are inflections with parallel tangent lines.}\label{fig:InfP}
\end{figure}

\begin{remark}

\
\noindent {\bf (1)} Notice that when $\alpha=1/2$, the sets $AEIL=AESS$ and $IPTL=MPTL$ meet at their singularities (see Proposition 2.4.9 of \cite{Holtom}). For $\alpha\neq1/2$, these sets are disjoint.
\medskip

\noindent {\bf (2)} It is possible to determine the first $\alpha=\alpha_0$ for which the cusps appear ($\alpha$-born) in $IPTL$ or $AEIL$. Indeed, in the case of the $IPTL$ (resp. $AEIL$), one can find the relation between the parameters, as mentioned in this section. Then, according to the regularity conditions given in Proposition \ref{prop:parallel} (resp. Theorem \ref{theo:non_Paralel}), the $\alpha$-born can be found effortlessly.
\end{remark}
\subsection{Coincident tangent lines }
Here, we study the limit case of the envelope of intermediate lines for the smooth plane curve $\gamma$. This means that we desire to investigate the behavior of the intermediate lines when one point tends to the other point. For such, consider the local parametrization of the curve $\gamma$ at each point, i.e. let
$$\gamma_1(t)=(t, f(t)) \text{  \space and \space }  \gamma_2(s)=(s, f(s)). $$ 

Let  $C=(t-s, f(t)-f(s)), M_\alpha = \left((1-\alpha)t+\alpha s, (1-\alpha)f(t) + \alpha f(s)\right), N_1(t)=(-f^{\prime}(t), 1)$  and $N_2(s)=(-f^{\prime}(s),1),$ be the chord, the intermediate point and the Euclidean normal vectors at $\gamma_1$ and $\gamma_2$, respectively.

Note that  $\langle N_i, U \rangle=[\gamma_i^{\prime}, U]$, for $U \in \mathbb{R}^2$ and $i=1,2$, where $\langle\cdot,\cdot\rangle$ denotes the usual inner product of $\mathbb{R}^2$. Therefore we have
\begin{align}\label{con2}
\nu_1(U)=\frac{[\gamma_1^{\prime}, U]}{[\gamma_1^{\prime}, \gamma_1^{\prime\prime}]^{\frac{1}{3}}}=\frac{\langle N_1, U \rangle}{[\gamma_1^{\prime}, \gamma_1^{\prime\prime}]^{\frac{1}{3}}}\text{\space and \space} \nu_2(U)=\frac{[\gamma_2^{\prime}, U]}{[\gamma_2^{\prime}, \gamma_2^{\prime\prime}]^{\frac{1}{3}}}=\frac{\langle N_2, U \rangle}{[\gamma_2^{\prime}, \gamma_2^{\prime\prime}]^{\frac{1}{3}}}. 
\end{align}

Thus, if we denote $\langle N_i, U \rangle = N_i(U)$, then the equation of the intermediate line becomes
$$(1-\alpha)N_2(C)N_1(X-M_\alpha)+\alpha N_1(C)N_2(X-M_\alpha)=0.  $$

From Theorem \ref{pre}, we obtain

$$\frac{N_1(C)  }{(f^{\prime\prime}(t_1))^{\frac{1}{3 }}}=-\lambda\frac{N_2(C)  }{(f^{\prime\prime}(t_2))^{\frac{1}{3 }}} \Rightarrow \frac{N_1(C)}{N_2(C)}=-\left(\frac{(1-\alpha)f^{\prime\prime}(t_1)}{\alpha f^{\prime\prime}(t_2)} \right)^{\frac{1}{3}}.   $$

Thus, the angular coefficient $A=A(s,t)$ of the intermediate line is given by


\begin{align*}
	  A(s,t)= \frac{(1-\alpha)f^{\prime}(t)(\alpha f^{\prime\prime}(s))^{\frac{1}{3}} - \alpha f^{\prime}(s)((1-\alpha)f^{\prime\prime}(t))^{\frac{1}{3}}                    }{(1-\alpha)(\alpha f^{\prime\prime}(s))^{\frac{1}{3}}-   \alpha((1-\alpha)f^{\prime\prime}(t))^{\frac{1}{3}}                        }.
\end{align*}
Therefore, for $\alpha=1/2$ we have
\begin{align*}
\lim_{s\longrightarrow t}A(s,t)= \frac{    \frac{1}{3}f'(t)(f''(t))^{-\frac{2}{3}}f'''(t) - f''(t)^\frac{4}{3} }{\frac{1}{3}(f''(t))^{-\frac{2}{3}}f'''(t)},
\end{align*} 
which is the affine normal of $\gamma$ at the point $p_1$.

For $\alpha\neq 1/2$, we have:
\begin{align*}
\lim_{s\longrightarrow t}A(s,t)= \frac{f'(t)(f''(t))^\frac{1}{3}((1-\alpha)\alpha^\frac{1}{3} -\alpha(1-\alpha)^\frac{1}{3})    }{(f''(t))^\frac{1}{3}((1-\alpha)\alpha^\frac{1}{3} -\alpha(1-\alpha)^\frac{1}{3})}=f'(t).
\end{align*}
Therefore, we have the following result:
\begin{theorem}
	Let $\gamma$ be a smooth plane curve, $p_1=\gamma_1(t)=(t,f(t))$ and $p_2=\gamma_2(s)=(s,f(s))$ two points of $\gamma$.
	When the parameter $s$ tends to $t$, then the intermediate line ($\alpha\neq 1/2$) tends to the tangent line of $\gamma$, at the point $p_1$.
\end{theorem}

\begin{figure}[htp]
\includegraphics[scale=0.4]{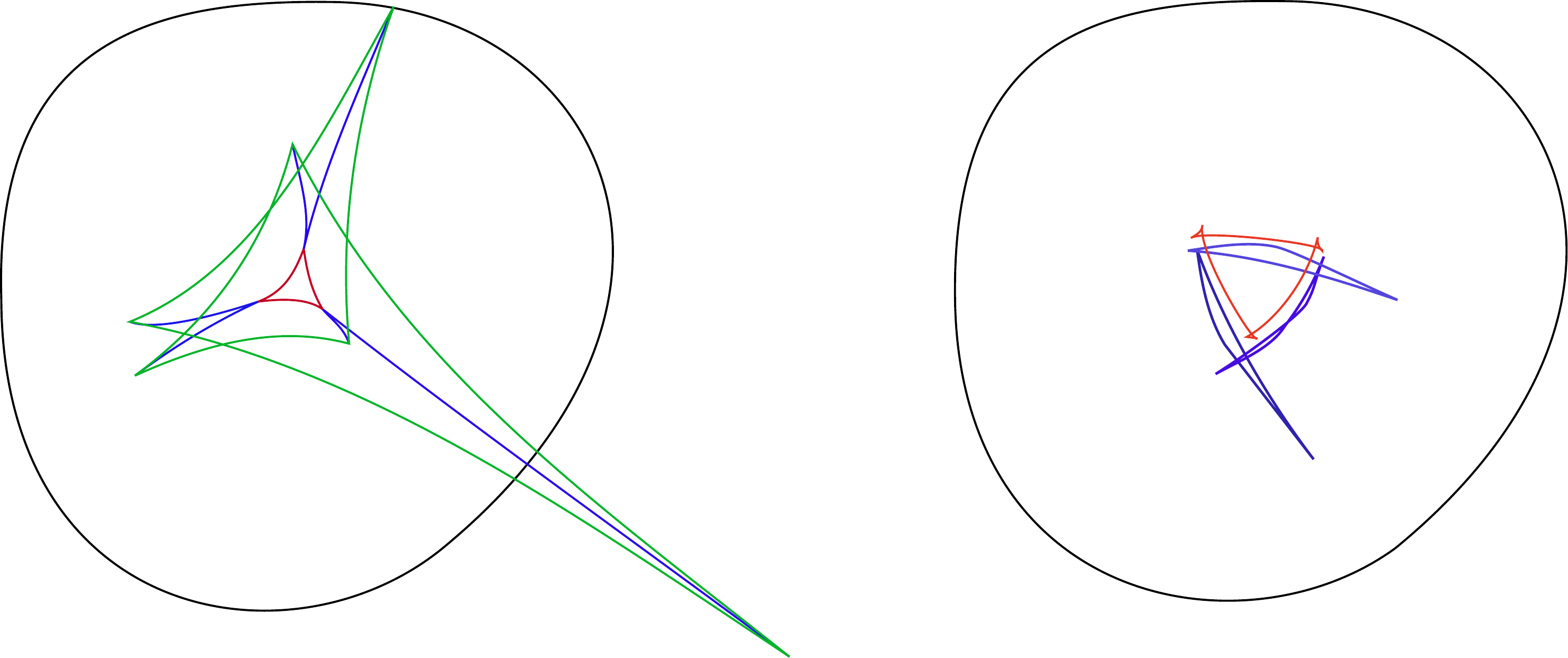}
\caption{A schematic example of Envelope of Intermediate Lines of the bean curve $\gamma(t)=\left(0.1\cos(2\pi t)+\cos(\pi t),0.1\sin(2\pi t+1)+\sin(\pi t)\right)$. The left figure illustrates the envelope when $\alpha=0.5$ where $MPTL$ (red), $AESS$ (blue) and Affine Evolute (green) appear. The right figure illustrates the envelope for $\alpha=0.6$, where $IPTL$ (red) and $AEIL$ (blue) appear.}
\end{figure}

\subsection{Singularity theory approach} 
 Our object in this section is not to study a single value of $\alpha$, but to study what happens to the envelope, as  $\alpha$ varies in $(0,1)\setminus\{1/2\}$.

Here, we shall state the results from the singularity theory, which allow us to make precise statements about the way in which the envelope of the family of intermediate lines evolves as $\alpha$ changes.
\begin{definition}[\cite{Giblin3}]
	For $X=(x,y,\alpha)$ in the discriminant surface \eqref{discriminant}, the function $f(t) = F(X, t)$ has
	\begin{enumerate}
\item type  $A_2$ at $t = t_0$ if $f^\prime(t_0) = f^{\prime\prime}(t_0) = 0, f^{\prime\prime\prime}(t_0) \neq 0$, \\
\item type $A_3$ at $t = t_0$ if $f^\prime(t_0) =  f^{\prime\prime}(t_0) =f^{\prime\prime\prime}(t_0)=0, f^{(4)}(t_0) \neq 0$.
	\end{enumerate}
\end{definition}
\begin{remark}
	The conditions given in Propositions \ref{prop:parallel},\ref{prop:parallel_p1_Inflection} and Theorem \ref{theo:non_Paralel} are the precise conditions for which the envelope has $A_2$ or $A_3$ singularity at the origin. 
\end{remark}
The additional required condition, besides $A_2$ or $A_3$ singularity, is presented bellow. 

\begin{definition}[\cite{giblin4}]
	Suppose that $t=t(s)$ and $F=F_s=0$ at $(X_0, s_0)=(x_0,y_0,\alpha_0,s_0)$. Also suppose that $f(s) = F(X_0, s)$ has an $A_r$-singularity at $s_0$. Consider the partial derivatives $F_{x}$, $F_{y}$ and $F_{\alpha}$ with respect to the parameters $x_i$, evaluated at $X_0$,
	their Taylor polynomials $\mathcal{T}_i$ up to degree $r-1$, expanded about $s_0$ (so these have $r$ terms). The family $F(X, s)$ is called a versal unfolding of $f$ at $s_0$ if the $\mathcal{T}_i$ span a vector space of dimension $r$. Thus, if the coefficients in the $\mathcal{T}_i$ are placed as the columns of an $(r-1) \times 3$ matrix, the rank
	is $r-1$. Obviously, this is possible only for $r \leq 3$.
\end{definition}
In the next result, we determine the versality condition of the family of the Affine Envelope of Intermediate Lines (i.e. $AEIL$) for an $A_2$-singularity at the origin. Analogously, one can find the versality conditions for the other cases. 
\begin{theorem} Let $\gamma$ be a closed plane curve, $p_1=\gamma(t)$ and $p_2=\gamma(s)$ two distintc points of $\gamma$ with concurrent tangent lines which have local parametrizations given in \eqref{P1andP2}. Consider $F$ as the family of intermediate lines of $\gamma$, given in \eqref{intermediate} . Suppose that $f(s)=F(x_0,y_0,\alpha_0,s)$ has an $A_2$-singularity at the origin. Then, $F$ is versal, if and only if 
		\[
		a_3=-\frac{1}{6}\frac{5\alpha-1}{\alpha^2b_1}.
		\] 
		
\end{theorem}
\begin{proof} Using the relation between the parameters given in \eqref{t} and the conditions of $A_2$ singularity, as mentioned in Theorem \ref{theo:non_Paralel}, we have
$$F(x,y,\alpha,s)=\alpha^2b_1^2(-\alpha^2b_1^2-b_1x+y)+(1-\alpha)\alpha b_1^2(-\alpha^2b_1^2+y) + O(s).$$

Notice that

\begin{eqnarray*}
j^1\left(\dfrac{\partial F}{\partial x}(X_0,s)\right)(0) & = & m_{11}s  \\ 
j^1\left(\dfrac{\partial F}{\partial y}(X_0,s)\right)(0) & = & m_{21}s \\
j^1\left(\dfrac{\partial F}{\partial \alpha}(X_0,s)\right)(0) & = & m_{31}s,  \\ 
\end{eqnarray*}
where $X_0=(x_0,y_0,\alpha_0)=(-\alpha b_1,0,\alpha)$ and

\begin{eqnarray*}
m_{11} & = & -\dfrac{\alpha b_1^2(6\alpha^2a_3b_1+5\alpha-1)}{2(3\alpha a_3b_1+\alpha+1)} \\
m_{12} & = & \dfrac{b_1(6\alpha^2a_3b_1+5\alpha-1)}{2(3\alpha a_3b_1+\alpha+1)} \\
m_{13} & = & \dfrac{\alpha b_1^3(36\alpha^3a_3^2b_1^2(\alpha-1)+12\alpha^3a_3b_1+9\alpha^2a_3b_1-24\alpha a_3b_1+\alpha^2+2\alpha-5)}{2(3\alpha a_3b_1+\alpha+1)}.
\end{eqnarray*}

The matrix $M=(m_{11},m_{12}, m_{13})$ has rank 1, if and only if $a_3=-\frac{1}{6}\frac{5\alpha-1}{\alpha^2b_1}.$
\end{proof}
 
Now let $X$ denote the parameters $(x,y,\alpha)$ and $X_0$, a fixed value of $X$. Suppose that the family of intermediate lines, $F$, be a versal family of $f(t)=F(X_0,t)$ where $f(t)$ has type $A_k$, $k=2$ or $3$. It is true that the discriminant of $F$, i.e. $\mathcal{D}_F$ given in (\ref{discriminant}), in a small neighborhood  of $X_0$, is locally diffeomorphic to the standard discriminant of a family of an $A_k$-singularity (details about these concepts are found in various places such as \cite{giblin4}). Figure \ref{fig:versal} illustrates the standard transitions of co-dimension one, which are called lips, beaks and swallowtail transitions.

\begin{figure}[htp]
	\includegraphics[scale=0.5]{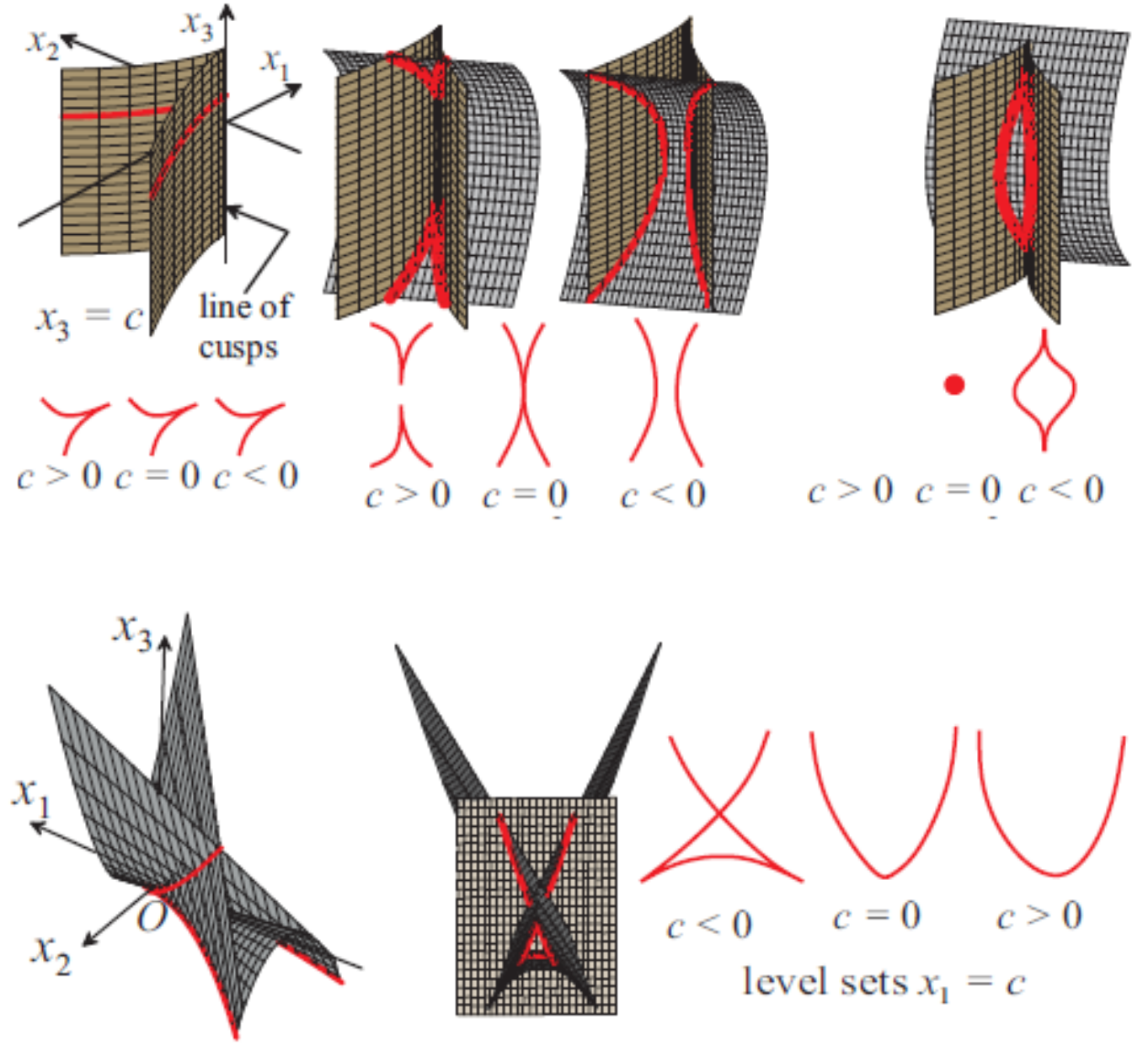}
	\caption{The standard transitions of cusps, beaks, lips (up left to right, respectively) and swallowtail (down) of a versal family $G=G(x_1,x_2,x_3,t)$ of a function of type $A_k$-singularity, $k=2$ or $3$. This Figure is given in \cite{Giblin3} page 7.}\label{fig:versal}
\end{figure} 
\medskip
\newpage
\bibliographystyle{amsplain}

\begin{thebibliography}{C-S-W}
	
	
	
	
	\bibitem{giblin4} J. W. Bruce and P. J. Giblin, {\it Curves and Singularities}. Cambridge University Press 1992, Second Edition, ISBN13: 978-0521429993.
	
	\bibitem{ady1} A. Cambraia and M. Craizer, {\it Envelope of mid-hyperplanes of a hypersurface.} Journal of Geometry, 108.3 (2017): 899-911.
	\bibitem{ady2} A. Cambraia and M. Craizer, {\it Envelope of Mid-Planes of a Surface and Some Classical Notions of Affine Differential Geometry.} Results in Mathematics 72.4 (2017): 1865-1880.
	
	\bibitem{giblin}  P. J. Giblin and G. Sapiro, {\it Affine invariant distances, envelopes and symmetry sets. }
	Geom. Dedicata, 71, 237-261, 1998.
	
	
	
	\bibitem{Giblin3} P. J. Giblin and J.P. Warder, {\it Evolving Evolutoids}. American Math. Monthly, 121, 871-889, 2014.
	
	
	
	\bibitem{Holtom} P. A. Holtom, {\it Affine-invariant symmetry sets.}
	Ph.D. Thesis, University of Liverpool, 2000.
	
	\bibitem{Nomizu}  K. Nomizu and T. Sasaki, {\it Affine Differential Geometry.}
	Cambridge University Press, 1994.
	
	
	\bibitem{Sano} T. Sano, {\it Bifurcations of affine invariants for one-parameter family of generic convex plane curves.} Banach Center Publications, vol 50, 1999.
	
	\bibitem{Buchin}  B. Su, {\it Affine Differential Geometry.} Gordon and Breach, 1983.
	%
	%
\bibitem{Warder}  J. P. Warder, {\it Symmetries of curves and surfaces.}
	Ph.D. Thesis, University of Liverpool, 2009.
	
	
	
	
\bibitem{Bruce}  J. W. Bruce, P. J. Giblin and C. G. Gibson, {\it Symmetry Sets.}
Proceedings of the Royal Society of Edinburgh, 101A, 163-186, 1985.
	
	
	
	
	
\end{thebibliography}

\end{document}